\documentclass[reqno]{amsart}
\usepackage[textheight=22cm,textwidth=15cm,headsep=1cm,centering]{geometry}
\usepackage{amsmath,amssymb,amsthm,esint,xcolor}
\usepackage[hidelinks]{hyperref}

\newtheorem{theorem}{Theorem}[section]
\newtheorem{lemma}[theorem]{Lemma}
\newtheorem{definition}[theorem]{Definition}

\newtheorem{remark}[theorem]{Remark}

\begin{document}

\title[Liouville theorems for stable solutions]
{Liouville theorems for stable solutions to the mean curvature equation}
\author[Xu]{Fanheng Xu}
\address{School of Mathematics (Zhuhai), Sun Yat-Sen University, 519082, Zhuhai, P. R. China}
\email{xufh7@mail.sysu.edu.cn}
\thanks{Xu was supported by the National Natural Science Foundation of China (No.12201654)}

\subjclass[2020]{Primary 35J93, 35B53; Secondary 35J62}
\keywords{Mean curvature equation, stable solution, Liouville theorem, Joseph--Lundgren exponent, Sobolev critical exponent}

\begin{abstract}
We study stable solutions of the mean curvature equation
$$
-\operatorname{div}\left(\frac{\nabla u}{\sqrt{1+|\nabla u|^2}}\right)
=|u|^{p-1}u
$$
in domains of $\mathbb R^n$. We establish Liouville-type results in three settings.
First, for globally stable solutions in $\mathbb R^n$, we prove that $u\equiv0$ if $1<p<p_{JL}$, where $p_{JL}$ is the Joseph--Lundgren exponent.
Second, for solutions in $\mathbb R^n$ that are stable outside a compact set, we show that if $n>2$ and $p=p_S$, where $p_S$ is the Sobolev critical exponent, then $u\equiv0$. Third, for solutions of the Dirichlet problem on connected unbounded domains with smooth boundary satisfying an epigraph-type geometric condition, we prove that if $n>2$, $1<p\le p_S$, and the solution is stable outside a compact set, then $u\equiv0$. The proofs are based on a gradient estimate for stable solutions, an integrability result, conservation laws and Pohozaev-type identities, and a unique continuation argument. These results extend Liouville-type theorems for stable solutions from the semilinear Lane--Emden equation to the mean curvature equation.
\end{abstract}

\maketitle

\section{Introduction}

A natural model problem in the study of critical exponent phenomena is the Lane--Emden equation
\begin{align}\label{semilinear equation}
-\Delta u=|u|^{p-1}u.
\end{align}
For this equation, two distinguished exponents play a fundamental role,
namely the Sobolev critical exponent and the Joseph--Lundgren exponent.
The Sobolev exponent is defined by
\begin{align*}
p_S=
\begin{cases}
+\infty, & n=2,\\
\dfrac{n+2}{n-2}, & n>2,
\end{cases}
\end{align*}
and is closely related to the scaling invariance of (\ref{semilinear equation}) and the critical Sobolev embedding.
In particular,
the exponent $p_S$ separates the subcritical and supercritical regimes from the point of view of the existence of positive entire solutions. The classical work of Gidas and Spruck \cite{Gidas Spruck CPAM 1981} established that if $n>2$ and
\begin{align*}
1<p<p_S,
\end{align*}
then every nonnegative solution of (\ref{semilinear equation}) in $\mathbb{R}^n$ is identically zero.
On the other hand,
positive entire solutions exist in the critical and supercritical regimes $p\ge p_S$.
At the critical exponent,
the positive solutions are completely classified by Caffarelli, Gidas and Spruck \cite{Caffarelli Gidas Spruck CPAM 1989},
while in the supercritical regime positive radial entire solutions can be obtained from the associated radial ODE.

A different critical exponent arises when stability is imposed.
This is the Joseph--Lundgren exponent
\begin{align*}
p_{JL}=
\begin{cases}
+\infty, & 2\leq n\leq10,\\
\dfrac{n^2-8n+4+8\sqrt{n-1}}{(n-2)(n-10)}, & n\geq11.
\end{cases}
\end{align*}
It first arose in the work of Joseph and Lundgren \cite{Joseph Lundgren ARMA 1973}
and is closely related to the stability of the singular radial solution
\[
u(x)=C|x|^{-\frac{2}{p-1}},
\]
whose linearized stability can be characterized via Hardy's inequality.
While $p_S$ reflects the scaling and Sobolev criticality of the equation,
$p_{JL}$ reflects the stability properties of singular solutions.
For $2\leq n\leq10$, $p_{JL}=+\infty$, whereas $p_S<+\infty$ for $n>2$;
for $n\geq11$, $p_S<p_{JL}$.
Consequently, in the intermediate range
\[
p_S<p<p_{JL},
\]
positive entire solutions may exist, while the stability condition rules out nontrivial stable solutions.

The critical role of $p_{JL}$ in rigidity and Liouville-type problems was highlighted in the seminal work of Farina \cite{Farina JMPA 2007},
who established optimal nonexistence results for stable solutions. Specifically, it was proved that for $1 < p < p_{JL}$,
any stable $C^2$ solution to (\ref{semilinear equation}) in $\mathbb{R}^n$ must be identically zero. Beyond entire solutions,
the Joseph--Lundgren exponent also characterizes the classification of solutions that are stable outside a compact set or possess a finite Morse index,
thereby laying a foundational framework for the modern theory of stable solutions in nonlinear elliptic equations.

The study of Liouville-type theorems for stable solutions has attracted extensive attention and has been generalized to various quasilinear and anisotropic operators in recent years.
Following the fundamental work of Farina on the classical Lane--Emden equation, Damascelli, Farina, Sciunzi, and Valdinoci identified a new critical exponent for the $m$-Laplace equation, establishing rigidity results for solutions that are stable or stable outside a compact set \cite{Damascelli Farina Sciunzi Valdinoci AIHP 2009}.
For the $p$-Laplacian, Farina, Mercuri and Willem \cite{Farina Mercuri Willem CVPDE 2019} proved classification results for problems on bounded and unbounded domains and study qualitative properties of sign-changing solutions involving critical nonlinearities.
Recently, Farina, Sciunzi and Vuono \cite{Farina Sciunzi Vuono NA 2024} successfully extended these Liouville-type theorems to quasilinear anisotropic elliptic equations, covering the pure Finsler $p$-Laplacian.
In a closely related geometric context, Farina, Sciunzi, and Valdinoci \cite{Farina Sciunzi Valdinoci ASNSPS 2008} developed a geometric approach to tackle Bernstein and De Giorgi type problems. Inspired by these profound developments in quasilinear and geometric settings, our present work aims to adapt these purely analytical techniques, particularly the delicate gradient estimates (see Lemma \ref{Lemma main gradient est}), to the mean curvature equation, where the non-polynomial nature of the operator introduces novel technical subtleties.

A natural quasilinear counterpart of the Lane--Emden equation is the mean curvature equation
\begin{align}\label{Mean Curvature Equation}
-\operatorname{div}\left(
\frac{\nabla u}{\sqrt{1+|\nabla u|^2}}
\right)
=|u|^{p-1}u
\qquad\text{in}\ \Omega\subset\mathbb R^n.
\end{align}
The operator
$\operatorname{div}( \nabla u / \sqrt{1+|\nabla u|^2} )$
is the prescribed mean curvature operator associated with the graph of $u$ in $\mathbb R^{n+1}$, up to the choice of sign and normalization.
Therefore, (\ref{Mean Curvature Equation}) may be viewed as a geometric counterpart of the Lane--Emden equation.

The variational structure of (\ref{Mean Curvature Equation}) is closely related to the geometry of graphs.
More precisely, consider the area-type functional
\begin{align}\label{Mean Curvature Equation Functional}
\mathcal E[u]
=
\int_\Omega
\left(
\sqrt{1+|\nabla u|^2}
-\frac{1}{p+1}|u|^{p+1}
\right).
\end{align}
Computing its first variation yields
\begin{align*}
\frac{d}{dt} \mathcal{E}[u+t\phi] \bigg|_{t=0}
= \int_\Omega \left(\frac{\nabla u \cdot \nabla \phi}{\sqrt{1+|\nabla u|^2}} - |u|^{p-1}u \phi\right).
\end{align*}
Thus, equating the first variation to zero and integrating by parts yields precisely equation (\ref{Mean Curvature Equation}).
This variational structure naturally leads to the notion of stability.
Indeed, the second variation of $\mathcal{E}$ at a critical point $u$ is given by
\[
\frac{d^2}{dt^2} \mathcal{E}[u+t\phi] \bigg|_{t=0}
= \int_\Omega
\left(
\frac{|\nabla\phi|^2}{\sqrt{1+|\nabla u|^2}} - \frac{(\nabla u \cdot \nabla\phi)^2}{(1+|\nabla u|^2)^{3/2}} - p|u|^{p-1} \phi^2
\right).
\]
If this quadratic form is nonnegative for every test function $\phi \in C_c^\infty(\Omega)$, then $u$ is called a stable critical solution. This motivates the following precise definition:
\begin{definition}[Stable solution]
Let $p>1$.
A function $u \in W_{loc}^{1,1}(\Omega)\cap L_{loc}^p(\Omega)$ is called a weak solution of (\ref{Mean Curvature Equation}) if
\begin{align}\label{mean curvature weak equation}
\int_\Omega \frac{\nabla u \cdot \nabla \varphi}{\sqrt{1+|\nabla u|^2}}
= \int_\Omega |u|^{p-1}u \varphi
\qquad \forall \ \varphi \in C_c^\infty(\Omega).
\end{align}
A stable solution is a weak solution satisfying the following stability inequality
\begin{align}\label{stable}
\int_\Omega p|u|^{p-1}\xi^2
\leq
\int_\Omega
\left(
\frac{|\nabla \xi|^2}{\sqrt{1+|\nabla u|^2}} - \frac{(\nabla u \cdot \nabla \xi)^2}{(1+|\nabla u|^2)^{3/2}}
\right)
\quad \forall \ \xi \in C_c^\infty(\Omega).
\end{align}
\end{definition}

Despite the difficulties caused by the non-uniform ellipticity,
Liouville-type results for the mean curvature operator have been obtained in several settings.
In particular,
Farina and Navarro \cite{Farina Navarro DCDS 2020} established Liouville-type results for stable radial solutions in $\mathbb R^n$.
However, the corresponding Liouville problem for general,
not necessarily radially symmetric,
stable solutions is substantially more delicate.

This naturally leads to the question of whether the critical-exponent phenomena described above for the Lane--Emden equation persist for general stable solutions of the mean curvature equation.
In this paper, we investigate these questions for general, not necessarily radial, solutions of the mean curvature equation. We first establish a Liouville theorem for globally stable solutions in $\mathbb R^n$. We then consider solutions that are stable outside a compact set and obtain Liouville-type results at the Sobolev critical exponent. Finally, we extend the analysis to a class of unbounded domains satisfying an epigraph-type geometric condition.

Our first result is stated as follows:
\begin{theorem}\label{Theorem stable}
Let $\Omega=\mathbb R^n$, $n\geq 2$, and let $u\in C^2(\mathbb{R}^n)$ be a stable solution of (\ref{Mean Curvature Equation}).
If
\begin{align*}
1<p<p_{JL},
\end{align*}
then $u\equiv 0$.
\end{theorem}

Theorem \ref{Theorem stable} can be viewed as a mean-curvature analogue of the milestone result of Farina \cite[Theorem 1]{Farina JMPA 2007} for the Lane--Emden equation.
Indeed, the range $1<p<p_{JL}$ obtained here is the same as that in the semilinear case.
In particular, the same Joseph--Lundgren exponent appears in the nonexistence range obtained here, despite the lack of uniform ellipticity and the absence of the scaling invariance enjoyed by the Lane--Emden equation.

Before presenting our next two results,
we recall that $u$ is stable outside a compact set $\mathcal{K}$ if stability condition (\ref{stable}) holds for every $\xi\in C_c^\infty(\Omega\setminus\mathcal{K})$.
This concept is particularly relevant in the study of solutions with finite Morse index.
Recall that the Morse index of a solution $u$ is defined as the maximal dimension of a subspace of $C_c^\infty(\Omega)$ on which the quadratic form associated with the linearized operator is negative definite.
For the quadratic form associated with the linearized operator, finite Morse index is equivalent to stability outside a compact subset of $\Omega$.
Therefore, stability outside a compact set provides a natural framework for studying finite Morse index solutions without requiring global stability.
In the sequel, we establish a Liouville-type theorem for solutions that are stable outside a compact set at the Sobolev critical exponent.

\begin{theorem}\label{Theorem stable outside compact}
Let $\Omega=\mathbb R^n$, $n>2$, and let $u \in C^2(\mathbb{R}^n)$ be a solution of (\ref{Mean Curvature Equation}).
Suppose that $u$ is stable outside a compact set $\mathcal K\subset\mathbb R^n$.
If
\begin{align*}
p = p_S,
\end{align*}
then $u\equiv 0$.
\end{theorem}

In the semilinear setting,
the corresponding classification for solutions stable outside a compact set was established by Farina \cite[Theorem 2]{Farina JMPA 2007}.
He proved that if $u\in C^2(\mathbb R^n)$ is stable outside a compact set, then $u\equiv 0$ provided
\begin{align*}
\begin{cases}
p>1, & n=2, \\
1<p<\infty,\quad p\neq p_S, & 3\leq n\leq10, \\
1<p<p_{JL},\quad p\neq p_S, & n\geq11.
\end{cases}
\end{align*}
For the Lane--Emden equation, this range is known to be sharp.
In particular, the Sobolev critical exponent $p_S$ is an exceptional value in the classification of solutions stable outside a compact set for the Lane--Emden equation.
By contrast, our theorem establishes nonexistence precisely at this critical exponent $p_S$ for the mean curvature equation.

It is worth recalling some existence and nonexistence results for the mean curvature equation.
For $n>2$ and in the range
\begin{align*}
1<p\leq\frac{n}{n-2},
\end{align*}
Bidaut-V'eron and Pohozaev \cite{Veron Pohozaev JAM 2001} obtained nonexistence results for nonnegative supersolutions of a broad class of quasilinear elliptic inequalities that includes the mean curvature equation.

On the other hand, del Pino and Guerra \cite{del Pino Guerra JDE 2007}
studied the existence of positive radial ground states.
Following the terminology in \cite{Ni Serrin RCMP 1985},
a ground state refers to a positive solution of (\ref{Mean Curvature Equation}) that vanishes at infinity, i.e.,
\begin{align*}
u(x)\to 0\quad\text{as}\quad |x|\to\infty.
\end{align*}
del Pino and Guerra proved that,
when $p$ lies below but sufficiently close to the Sobolev critical exponent $p_S$,
the mean curvature equation admits a large,
finite number of radial ground states with fast decay.
Thus, their result provides the existence of nontrivial positive entire solutions in a portion of the intermediate range,
but does not cover the whole interval
\begin{align}\label{intermediate range}
\frac{n}{n-2}<p<\frac{n+2}{n-2}.
\end{align}
Consequently, the complete existence and nonexistence picture in this intermediate regime remains unclear. In particular, it is natural to ask whether nontrivial entire solutions exist throughout the whole range (\ref{intermediate range}).

Next, we consider domains satisfying
\begin{align}\label{domain eg epigraphs}
\exists X \in \mathbb{R}^n,\ |X|=1,\ \text{such that}\ \nu(x)\cdot X\geq 0,\ \nu(x)\cdot X\not\equiv 0\ \text{on}\ \partial\Omega,
\end{align}
where $\nu$ denotes the unit outward normal to $\partial\Omega$.
This geometric condition, originally introduced by Esteban and Lions \cite{Esteban Lions PRSE 1982} in their study of semilinear elliptic problems in unbounded domains, is satisfied by a large class of domains, including half-spaces and more generally domains that are the epigraphs of functions.
Our result is as follows:
\begin{theorem}\label{Theorem epigraph}
Let $\Omega \subset \mathbb{R}^n$ be a connected domain with smooth boundary satisfying (\ref{domain eg epigraphs}), $n>2$, and let $u \in C^2(\overline{\Omega})$ be a solution of the Dirichlet problem:
\begin{align*}
\left\{
\begin{array}{ll}
-\operatorname{div}\left(\dfrac{\nabla u}{\sqrt{1+|\nabla u|^2}}\right) = |u|^{p-1}u
&\qquad \text{in}\ \Omega, \\
u=0 &\qquad \text{on}\ \partial \Omega,
\end{array}
\right.
\end{align*}
Suppose that $u$ is stable outside a compact set $\mathcal{K}\subset \Omega$.
If
\begin{align*}
1<p\leq p_S,
\end{align*}
then $u\equiv 0$.
\end{theorem}

For the semilinear Dirichlet problem on domains satisfying (\ref{domain eg epigraphs}), an analogous rigidity property was obtained in \cite[Theorem 9(a)]{Farina JMPA 2007}.
In that setting,
Farina proved that if $u \in C^2(\overline{\Omega})$ is stable outside a compact subset of $\Omega$, then $u\equiv 0$ provided $1<p\leq p_S$.
Thus, the range $1<p\leq p_S$ in Theorem \ref{Theorem epigraph} coincides with the corresponding range in Farina's semilinear result.

For the Lane--Emden equation, the corresponding exponent ranges are known to be optimal, as the failure of the Liouville property beyond the critical range can be demonstrated by the existence of suitable nontrivial solutions.
For the mean curvature equation,
however,
it remains open whether the nonexistence ranges in our Liouville theorems are optimal.
This difficulty is closely related to the quasilinear nature,
nonlinearity,
and lack of scaling invariance of the mean curvature operator,
which substantially complicate both the construction of nontrivial solutions and the derivation of sharp a priori estimates.

The remainder of the paper is organized as follows. In Section 2, we establish several auxiliary results that will be used throughout the proofs, including a gradient estimate for stable solutions, an integrability lemma, and conservation laws and Pohozaev-type identities. Section 3 is devoted to the proofs of the three main theorems.

We denote by \(C\) a generic positive constant, which may change from line to line;
when dependence on parameters is relevant, we write \(C_{n,p,a}\), etc.
For \(R>0\), \(B_R\) stands for the open ball of radius \(R\) centered at the origin in \(\mathbb R^n\),
and \(B_R(x_0)\) for the ball centered at \(x_0\).
We use $I$ for the identity matrix and $\delta_{ij}$ for the Kronecker delta.
All integrals are with respect to the Lebesgue measure in $\mathbb{R}^n$ and we usually omit $dx$.

\section{Gradient Estimate and Pohozaev Identity}

\subsection{Gradient Estimate}

We introduce a gradient estimate for the mean curvature equation,
which is crucial for the proof of Theorem \ref{Theorem stable}.

\begin{lemma}\label{Lemma main gradient est}
Let $p>1$ and
\begin{align*}
a_+(p) := -1+2p+2\sqrt{p(p-1)}.
\end{align*}
Suppose $u\in C^2(\Omega)$ is a stable solution of (\ref{Mean Curvature Equation}).
Then for every
\begin{align*}
a\in\left[1, a_+(p)\right),
\end{align*}
there exists an integer $m = m(p,a)>0$ and a constant $C=C(n,p,a)>0$ such that
\begin{align*}
\int_\Omega |u|^{p+a} \phi^{2m}
+ \int_\Omega |u|^{a-1} \frac{|\nabla u|^2}{\sqrt{1+|\nabla u|^2}}\phi^{2m}
\leq C \int_\Omega |\nabla\phi|^{\frac{2(p+a)}{p-1}}
\end{align*}
for all $\phi\in C_c^\infty(\Omega)$ with $0\leq \phi\leq 1$.
\end{lemma}

\begin{remark}
Our gradient estimate differs structurally from the corresponding semilinear inequality \cite[Proposition 1]{Farina JMPA 2007}.
While Farina's estimate contains the second-order quantity $|\Delta\psi|$ on the right-hand side,
ours involves only the first-order term $|\nabla\phi|$.
In Euclidean space this difference is immaterial for the usual radial cut-offs.
On a Riemannian manifold, however, one naturally takes $\phi=\varphi(d(\,\cdot\,,o)/R)$ where $d$ is the distance function.
Controlling $\Delta\phi$ requires a Laplacian comparison estimate for the distance function, which in turn requires suitable geometric assumptions such as a lower Ricci-curvature bound.
In contrast, our estimation only requires an estimate of $|\nabla\phi|$, and when extended to geometric setting, no additional curvature assumptions are needed.
\end{remark}

\begin{proof}[Proof of Lemma \ref{Lemma main gradient est}]
Let $\eta \in C_c^{\infty}(\Omega)$ be a nonnegative cut-off function.
Since $a\ge1$, the functions $t\mapsto |t|^{a-1}t$ belongs to $C^1(\mathbb R)$.
Hence
\begin{align*}
\varphi = |u|^{a-1} u \eta^2 \in C_c^1(\Omega).
\end{align*}
By a standard density argument,
the weak formulation (\ref{mean curvature weak equation}) remains valid for $C_c^1$ test functions.
Therefore substituting $\varphi$ and its gradient
\begin{align*}
\nabla \varphi = a |u|^{a-1} \eta^2 \nabla u + 2|u|^{a-1}u \eta\nabla \eta
\end{align*}
into (\ref{mean curvature weak equation}) gives
\begin{align}\label{weak eq est 1}
a\int_\Omega \frac{|u|^{a-1}|\nabla u|^2\eta^2}{\sqrt{1+|\nabla u|^2}}
+ 2\int_\Omega |u|^{a-1}u \frac{\nabla u \cdot \nabla\eta}{\sqrt{1+|\nabla u|^2}}\eta
= \int_\Omega |u|^{p+a} \eta^2.
\end{align}

By the same density argument, the function
\begin{align*}
\xi = |u|^{\frac{a-1}{2}} u \eta
\end{align*}
is admissible in the stability inequality (\ref{stable}).
We compute
\begin{align*}
\nabla \xi
= \frac{1+a}{2} |u|^{\frac{a-1}{2}} \eta \nabla u + |u|^{\frac{a-1}{2}} u \nabla\eta,
\end{align*}
\begin{align*}
|\nabla \xi|^2
= \left(\frac{1+a}{2}\right)^2 |u|^{a-1} |\nabla u|^2 \eta^2
+ |u|^{1+a} |\nabla\eta|^2
+ (a+1)|u|^{a-1}u\eta (\nabla u\cdot\nabla\eta),
\end{align*}
and
\begin{align*}
\left(\nabla u \cdot \nabla \xi\right)^2
= \left(\frac{1+a}{2}\right)^2 |u|^{a-1} |\nabla u|^4 \eta^2
+ |u|^{1+a} \left(\nabla u \cdot \nabla\eta\right)^2
+ (a+1) |u|^{a-1}u |\nabla u|^2 \eta (\nabla u\cdot \nabla\eta)
\end{align*}
Inserting these into (\ref{stable}) yields
\begin{align}\label{7e71ce28a}
p \int_\Omega|u|^{p+a} \eta^2
&\leq \left(\frac{1+a}{2}\right)^2 \int_\Omega \frac{|u|^{a-1} |\nabla u|^2 \eta^2}{(1+|\nabla u|^2)^{3/2}}
+ \int_\Omega\frac{|u|^{1+a}|\nabla\eta|^2}{\sqrt{1+|\nabla u|^2}} \notag \\
&\qquad - \int_\Omega\frac{|u|^{1+a}\left(\nabla u \cdot \nabla\eta\right)^2}{(1+|\nabla u|^2)^{3/2}}
+(a+1)\int_\Omega\frac{|u|^{a-1}u \left(\nabla u \cdot \nabla\eta\right)\eta}{(1+|\nabla u|^2)^{3/2}}
\end{align}
Using
\[
\frac{1}{(1+|\nabla u|^2)^{3/2}}
\le
\frac{1}{\sqrt{1+|\nabla u|^2}},
\]
we enlarge the first term on the right-hand side of
(\ref{7e71ce28a}) by
\[
\left(\frac{1+a}{2}\right)^2
\int_\Omega
\frac{|u|^{a-1}|\nabla u|^2\eta^2}
{\sqrt{1+|\nabla u|^2}}.
\]
Next, we add the term
\begin{align*}
\varepsilon \left(\frac{a+1}{2}\right)^2
\int_\Omega \frac{|u|^{a-1}|\nabla u|^2\eta^2}{\sqrt{1+|\nabla u|^2}}
\end{align*}
to both sides, where $\varepsilon >0$ is to be chosen later.
We obtain
\begin{align}\label{stable est 1}
&p \int_\Omega|u|^{p+a} \eta^2
+ \varepsilon \left(\frac{1+a}{2}\right)^2 \int_\Omega\frac{|u|^{a-1}|\nabla u|^2\eta^2}{\sqrt{1+|\nabla u|^2}} \notag\\
&\leq (1+\varepsilon)\left(\frac{1+a}{2}\right)^2 \int_\Omega\frac{|u|^{a-1} |\nabla u|^2 \eta^2}{\sqrt{1+|\nabla u|^2}}
+ \int_\Omega\frac{|u|^{1+a}|\nabla\eta|^2}{\sqrt{1+|\nabla u|^2}} \notag\\
&\qquad - \int_\Omega\frac{|u|^{1+a}\left(\nabla u \cdot \nabla\eta\right)^2}{(1+|\nabla u|^2)^{3/2}}
+(a+1)\int_\Omega |u|^{a-1}u \eta\frac{\left(\nabla u \cdot \nabla\eta\right)}{(1+|\nabla u|^2)^{3/2}}
\end{align}
Using (\ref{weak eq est 1}) to eliminate the first term on the right-hand side of (\ref{stable est 1}), we obtain
\begin{align}\label{071a9e4d4}
&\left(p - \frac{1+\varepsilon}{a} \left(\frac{1+a}{2}\right)^2\right) \int_\Omega |u|^{p+a}\eta^2
+ \varepsilon \left(\frac{1+a}{2}\right)^2 \int_\Omega |u|^{a-1} \frac{|\nabla u|^2}{\sqrt{1+|\nabla u|^2}} \eta^2 \notag\\
&\leq \int_\Omega \frac{|u|^{1+a}|\nabla\eta|^2}{\sqrt{1+|\nabla u|^2}}
- \int_\Omega \frac{|u|^{1+a}\left(\nabla u \cdot \nabla\eta\right)^2}{(1+|\nabla u|^2)^{3/2}} \notag\\
&\qquad +\int_\Omega \left(\frac{1+a}{(1+|\nabla u|^2)^{3/2}}
- \frac{2(1+\varepsilon)\left(\frac{1+a}{2}\right)^2}{a\sqrt{1+|\nabla u|^2}}\right)|u|^{a-1} u (\nabla u \cdot \nabla\eta)\eta
\end{align}
Under the condition $a>0$, we have
\begin{align*}
p - \frac{1}{a} \left(\frac{1+a}{2}\right)^2>0
&\Longleftrightarrow 4ap >(a+1)^2 \\
&\Longleftrightarrow a^2 + (2-4p)a +1 < 0 \\
&\Longleftrightarrow a \in (a_-(p), a_+(p)),
\end{align*}
where
\begin{align*}
a_-(p) &:= -1+2p-2\sqrt{p(p-1)}.
\end{align*}
Noting $a_-(p) = \left(\sqrt{p} - \sqrt{p-1}\right)^2 <1$,
and recalling that we are working in the sub-interval $a \in [1,a_+(p)) \subset (a_-(p), a_+(p))$,
which means that
\begin{align*}
p - \frac{1}{a} \left(\frac{1+a}{2}\right)^2>0.
\end{align*}
Therefore the coefficient of the first integral on the left of (\ref{071a9e4d4}) is positive for every sufficiently small $\varepsilon>0$.
Fix such a small enough $\varepsilon$.
Consequently,
after absorbing all constants into a generic constant $C=C_{n,p,a}$,
from (\ref{071a9e4d4}) we obtain
\begin{align}\label{est 1}
\int_\Omega |u|^{p+a}\eta^2
+ \int_\Omega |u|^{a-1} \frac{|\nabla u|^2}{\sqrt{1+|\nabla u|^2}} \eta^2
\leq C_{n,p,a} \int_\Omega \left(\frac{|u|^{1+a}|\nabla\eta|^2}{\sqrt{1+|\nabla u|^2}}
+ \frac{|u|^{a} |\nabla u| |\nabla \eta|\eta}{\sqrt{1+|\nabla u|^2}}\right).
\end{align}
Set $\eta = \phi^m$, where \(m\) is a positive integer to be chosen later.
Then (\ref{est 1}) becomes
\begin{align}\label{est 2}
&\int_\Omega |u|^{p+a}\phi^{2m}
+ \int_\Omega |u|^{a-1} \frac{|\nabla u|^2}{\sqrt{1+|\nabla u|^2}} \phi^{2m} \notag\\
&\leq C_{n,p,a,m} \int_\Omega \left(\frac{|u|^{1+a}\phi^{2(m-1)}|\nabla\phi|^2}{\sqrt{1+|\nabla u|^2}}
+ \frac{|u|^{a} |\nabla u| \phi^{2m-1}|\nabla\phi|}{\sqrt{1+|\nabla u|^2}}\right).
\end{align}
Apply the following Young's inequality
\begin{align*}
XY \leq \delta X^{\frac{p+a}{1+a}} + C_\delta Y^{\frac{p+a}{p-1}}
\end{align*}
where $\delta$ is a small positive constant chosen later.
Taking
\begin{align*}
X = |u|^{1+a}\phi^{2(m-1)}
\quad \text{and} \quad Y = |\nabla\phi|^2,
\end{align*}
we get
\begin{align*}
\int_\Omega \frac{|u|^{1+a} \phi^{2(m-1)} |\nabla\phi|^2}{\sqrt{1+|\nabla u|^2}}
&\leq \int_\Omega |u|^{1+a}\phi^{2(m-1)}|\nabla\phi|^2 \notag\\
&\leq \delta \int_\Omega |u|^{p+a} \phi^{\frac{2(m-1)(p+a)}{1+a}}
+ C_\delta \int_\Omega |\nabla\phi|^{\frac{2(p+a)}{p-1}}
\end{align*}
Choose integer $m$ so large that $m \geq (p+a)/(p-1)$, which ensures that
\begin{align*}
\frac{2(m-1)(p+a)}{1+a} \geq 2m.
\end{align*}
Since $0 \leq \phi \leq 1$, we deduce
\begin{align}\label{after Youngs eq 1}
\int_\Omega \frac{|u|^{1+a} \phi^{2(m-1)} |\nabla\phi|^2}{\sqrt{1+|\nabla u|^2}}
&\leq \delta \int_\Omega |u|^{p+a} \phi^{2m} + C_\delta \int_\Omega |\nabla\phi|^{\frac{2(p+a)}{p-1}}
\end{align}
Also using the three-factor Young's inequality
\begin{align*}
XYZ \leq \delta X^2 + \delta Y^\frac{2(p+a)}{1+a} + C_\delta Z^\frac{2(p+a)}{p-1},
\end{align*}
and taking
\begin{align*}
X=\frac{|u|^{\frac{a-1}{2}}|\nabla u|\phi^{m}}{(1+|\nabla u|^2)^{1/4}},
\quad Y=|u|^{\frac{a+1}{2}} \phi^{m-1}
\quad\text{and}\quad Z=|\nabla\phi|,
\end{align*}
then
\begin{align}\label{after Youngs eq 2}
&\int_\Omega \frac{|u|^{a} |\nabla u|\phi^{2m-1} |\nabla\phi|}{\sqrt{1+|\nabla u|^2}} \notag\\
&\leq \int_\Omega \frac{|u|^{a} |\nabla u|\phi^{2m-1} |\nabla\phi|}{(1+|\nabla u|^2)^{1/4}} \notag\\
&\leq \delta \int_\Omega |u|^{a-1} \frac{|\nabla u|^2}{\sqrt{1+|\nabla u|^2}} \phi^{2m}
+\delta \int_\Omega |u|^{p+a} \phi^{\frac{2(m-1)(p+a)}{1+a}}
+C_\delta \int_\Omega |\nabla\phi|^{\frac{2(p+a)}{p-1}} \notag\\
&\leq \delta \int_\Omega |u|^{a-1} \frac{|\nabla u|^2}{\sqrt{1+|\nabla u|^2}} \phi^{2m}
+ \delta\int_\Omega |u|^{p+a} \phi^{2m}
+C_\delta \int_\Omega |\nabla\phi|^{\frac{2(p+a)}{p-1}}
\end{align}
Let $C_0=C_{n,p,a,m}$ denote the constant in (\ref{est 2}).
Choose $\delta>0$ sufficiently small so that
\[
3C_0\delta<1.
\]
Inserting (\ref{after Youngs eq 1}) and (\ref{after Youngs eq 2}) into (\ref{est 2}) yields
\begin{align*}
\int_\Omega |u|^{p+a} \phi^{2m}
+ \int_\Omega |u|^{a-1} \frac{|\nabla u|^2}{\sqrt{1+|\nabla u|^2}} \phi^{2m}
\leq C_{n,p,a} \int_\Omega |\nabla\phi|^{\frac{2(p+a)}{p-1}},
\end{align*}
which is the desired estimate.
This proves Lemma \ref{Lemma main gradient est}.
\end{proof}

The following variant of the preceding gradient estimate is useful when homogeneous Dirichlet boundary conditions are imposed. In this case, the vanishing of $u$ on $\partial\Omega$ allows us to use cut-off functions defined on $\mathbb R^n\setminus \mathcal{K}$, which will be convenient in the sequel.

\begin{lemma}\label{Lemma main gradient est Dirichlet zero boundary}
Let $p>1$ and let $\Omega \subset \mathbb{R}^n$ be a domain with smooth boundary.
Suppose $u\in C^2(\overline{\Omega})$ is a solution of
\begin{align*}
\left\{
\begin{array}{ll}
-\operatorname{div}\left(\dfrac{\nabla u}{\sqrt{1+|\nabla u|^2}}\right) = |u|^{p-1}u
&\qquad \text{in}\ \Omega, \\
u=0 &\qquad \text{on}\ \partial \Omega
\end{array}
\right.
\end{align*}
that is stable outside a compact set $\mathcal{K}\subset \Omega$.
Then, for every
\begin{align*}
a\in[1,a_+(p)),
\end{align*}
there exist an integer $m=m(p,a)>0$ and a constant
$C=C_{n,p,a}>0$ such that
\begin{align*}
\int_\Omega |u|^{p+a}\phi^{2m}
+
\int_\Omega
|u|^{a-1}
\frac{|\nabla u|^2}{\sqrt{1+|\nabla u|^2}}
\phi^{2m}
\le
C\int_\Omega
|\nabla\phi|^{\frac{2(p+a)}{p-1}}
\end{align*}
for every $\phi\in C_c^\infty(\mathbb R^n\setminus \mathcal{K})$ satisfying
$0\le\phi\le1$.
Here $a_+(p)$ is defined as in Lemma \ref{Lemma main gradient est}.
\end{lemma}

\begin{proof}
The proof is the same as that of Lemma \ref{Lemma main gradient est}, except that the cut-off function $\phi$ is not required to vanish near $\partial\Omega$. The homogeneous Dirichlet condition allows the test functions below to be approximated by functions in $C_c^\infty(\Omega)$.
Indeed, since $u\in C^2(\overline{\Omega})$,
$\phi\in C_c^\infty(\mathbb R^n \setminus \mathcal{K})$,
and $a\ge1$,
both functions
\begin{align*}
|u|^{a-1}u\phi^{2m}
\quad\text{and}\quad
|u|^{\frac{a-1}{2}}u\phi^m
\end{align*}
belong to $H^1(\Omega)$. Since $u=0$ on $\partial\Omega$ and $\phi$ is compactly supported in $\mathbb R^n\setminus\mathcal{K}$, these functions vanish both on $\partial\Omega$ (in the sense of trace) and in a neighborhood of $\mathcal{K}$. Hence they actually belong to $H_0^1(\Omega\setminus\mathcal{K})$.
Therefore, by approximation with functions in $C_c^\infty(\Omega\setminus\mathcal K)$,
they may be used as test functions in both the weak formulation and the stability inequality.
With these test functions, the computations in the proof of Lemma \ref{Lemma main gradient est} are unchanged.
\end{proof}

\begin{lemma}\label{Lemma integrable}
Suppose that either
\begin{itemize}
\item[(i)] $\Omega=\mathbb R^n$, and $u\in C^2(\mathbb R^n)$ is a solution that is stable outside a compact set $\mathcal K\subset\mathbb R^n$; or
\item[(ii)] $\Omega\subset\mathbb R^n$ is a domain with smooth boundary, $u\in C^2(\overline{\Omega})$ is a solution of
\begin{align*}
\left\{
\begin{array}{ll}
-\operatorname{div}\left(\dfrac{\nabla u}{\sqrt{1+|\nabla u|^2}}\right) = |u|^{p-1}u
&\qquad \text{in}\ \Omega, \\
u=0 &\qquad \text{on}\ \partial \Omega,
\end{array}
\right.
\end{align*}
that is stable outside a compact set $\mathcal K\subset\Omega$.
\end{itemize}
If
\begin{align*}
1<p\le p_S,
\end{align*}
then there exists $r_0>0$ such that
\begin{align*}
\int_{\Omega\setminus B_{r_0}} |u|^{p+1}<\infty
\quad \text{and} \quad
\int_{\Omega\setminus B_{r_0}}
\frac{|\nabla u|^2}{\sqrt{1+|\nabla u|^2}}<\infty.
\end{align*}
\end{lemma}

\begin{proof}
In either case,
we apply the corresponding gradient estimate,
namely Lemma \ref{Lemma main gradient est} in case (i) and Lemma \ref{Lemma main gradient est Dirichlet zero boundary} in case (ii),
with $a=1$.
Choose $r_0>0$ such that $\mathcal K\subset B_{r_0}$, and let $\phi_R\in C_c^\infty(\mathbb R^n \setminus \mathcal{K})$ be a radial cut-off function satisfying
\begin{align*}
0\le \phi_R\le1,\qquad
\phi_R=0\quad\text{in}\ B_{r_0},
\qquad
\phi_R=1\quad\text{in}\ B_R\setminus B_{r_0+1},
\qquad
\phi_R=0\quad\text{outside}\ B_{2R},
\end{align*}
with
\begin{align*}
|\nabla\phi_R|\le C
\quad\text{in}\ B_{r_0+1}\setminus B_{r_0},
\qquad
|\nabla\phi_R|\le \frac{C}{R}
\quad\text{in}\ B_{2R}\setminus B_R.
\end{align*}

In both cases, $\phi_R$ is admissible for the corresponding gradient estimate.
Indeed, in case (i),
although $u$ is only stable outside $\mathcal K$,
the cut-off function $\phi_R$ is supported in $\mathbb R^n\setminus B_{r_0}\subset \mathbb R^n\setminus \mathcal K$,
so the stability inequality is available for the test functions used in the proof of Lemma \ref{Lemma main gradient est}.
In case (ii), we directly apply Lemma \ref{Lemma main gradient est Dirichlet zero boundary}.
Setting
\begin{align*}
q:=\frac{2(p+1)}{p-1},
\end{align*}
we have $q\ge n$ for $1<p\le p_S$.
Hence the same gradient estimate holds,
\begin{align*}
\int_\Omega|\nabla\phi_R|^q
\le C+C R^{n-q}
\le C,
\end{align*}
where $C$ is independent of $R$.
Therefore, in either case, the corresponding gradient estimate gives
\begin{align*}
\int_{\Omega\cap(B_R\setminus B_{r_0+1})}|u|^{p+1}
+
\int_{\Omega\cap(B_R\setminus B_{r_0+1})}
\frac{|\nabla u|^2}{\sqrt{1+|\nabla u|^2}}
\le C.
\end{align*}
Letting $R\to\infty$, we obtain
\begin{align*}
\int_{\Omega\setminus B_{r_0+1}}|u|^{p+1} < \infty
\qquad \text{and} \qquad
\int_{\Omega\setminus B_{r_0+1}}
\frac{|\nabla u|^2}{\sqrt{1+|\nabla u|^2}} < \infty.
\end{align*}
Replacing $r_0+1$ by $r_0$ completes the proof.
\end{proof}

\subsection{Conservation law and Pohozaev-type identity}

The Pohozaev-type identity used below can be systematically understood from the perspective of stress--energy tensor conservation laws,
following the influential framework developed by Esteban and Lions \cite{Esteban Lions PRSE 1982} for semilinear elliptic problems in unbounded domains.
To motivate this, we briefly recall this connection for the classical Lane--Emden equation
\begin{align*}
-\Delta u=f(u)
\qquad\text{in}\ \Omega,
\end{align*}
and let
\begin{align*}
F(u)=\int_0^u f(s)\,ds.
\end{align*}
The equation is the Euler--Lagrange equation associated with the functional
\begin{align*}
\mathcal E_0[u]=\int_\Omega\left(\frac12|\nabla u|^2-F(u)\right).
\end{align*}
As systematically exploited in \cite{Esteban Lions PRSE 1982}, the associated stress--energy tensor given by
\begin{align*}
T=\nabla u \otimes \nabla u
-\left(\frac12|\nabla u|^2-F(u)\right)I
\end{align*}
satisfies the conservation law $\operatorname{div}T=0$ for any solution. Contracting this tensor with vector fields (such as constant vectors or the dilation vector field) and applying the divergence theorem yields the corresponding integral identities. More generally, such systematic variational identities and conservation laws were further developed for broad classes of variational problems by Pucci and Serrin \cite{Pucci Serrin Indiana 1986}.

A similar structure is present for the mean curvature equation.
The functional (\ref{Mean Curvature Equation Functional}) has (\ref{Mean Curvature Equation}) as its Euler--Lagrange equation.
The associated stress--energy tensor is
\begin{align*}
T = \frac{\nabla u \otimes \nabla u}{\sqrt{1+|\nabla u|^2}} - \left(\sqrt{1+|\nabla u|^2} -1 -\frac{1}{p+1}|u|^{p+1}\right)I,
\end{align*}
which likewise satisfies $\operatorname{div}T=0$.
Thus, the Pohozaev identities for the mean curvature equation arise from the same underlying conservation law. Related identities have previously been obtained for the mean curvature operator, particularly in the radial setting; see, for instance,
\cite{Ni Serrin CPAM 1986, Filippucci Zheng MathAnn 2024}.
We state the conservation law as follows:
\begin{lemma}[Conservation law for the mean curvature equation]\label{Lemma conservation law}
Let $f\in C^1(\mathbb{R})$ and let $u\in C^2(\Omega)$ be a solution of
\begin{align}\label{Mean Curvature Equation f(u)}
-\operatorname{div} \left(\frac{\nabla u}{\sqrt{1+|\nabla u|^2}}\right) = f(u)
\qquad \text{in}\ \Omega.
\end{align}
Let
\begin{align*}
F(t) = \int_0^t f(s) ds.
\end{align*}
Define the tensor field $T=(T_{ij})$ by
\begin{align*}
T_{ij} = \frac{u_i u_j}{\sqrt{1+|\nabla u|^2}} - \left(\sqrt{1+|\nabla u|^2} - 1 - F(u)\right)\delta_{ij}.
\end{align*}
Then
\begin{align*}
\operatorname{div} T = 0.
\end{align*}
\end{lemma}

\begin{proof}
It suffices to show that
\begin{align*}
\sum_{j=1}^n \partial_j T_{ij} = 0 \qquad \text{for}\ i=1,2,\cdots, n.
\end{align*}
By the equation (\ref{Mean Curvature Equation f(u)}), we have
\begin{align*}
\sum_j \partial_j \left(\frac{u_i u_j}{\sqrt{1+|\nabla u|^2}}\right)
&= \sum_j \frac{u_{ij} u_j}{\sqrt{1+|\nabla u|^2}}
+ u_i \sum_j \partial_j\left(\frac{u_j}{\sqrt{1+|\nabla u|^2}}\right) \notag\\
&=\frac{\partial_i \left(|\nabla u|^2\right)}{2\sqrt{1+|\nabla u|^2}}
-u_i f(u)
\end{align*}
On the other hand, we have
\begin{align*}
\sum_{j=1}^n\partial_j
\left[
\left(
\sqrt{1+|\nabla u|^2}-1-F(u)
\right)\delta_{ij}
\right]
&=
\partial_i
\left(
\sqrt{1+|\nabla u|^2}-1-F(u)
\right)\\
&=
\frac{\partial_i \left(|\nabla u|^2\right)}
{2\sqrt{1+|\nabla u|^2}}
-f(u)u_i,
\end{align*}
where we have used $F'=f$.
Subtracting the two identities gives
\begin{align*}
\sum_j \partial_j T_{ij}
= \sum_j \partial_j \left(\frac{u_i u_j}{\sqrt{1+|\nabla u|^2}}\right)
-\sum_j \partial_j \left(\left(\sqrt{1+|\nabla u|^2} -1 -F(u)\right) \delta_{ij}\right) = 0.
\end{align*}
Since $i=1,\dots,n$ was arbitrary, we conclude that
\begin{align*}
\operatorname{div}T=0
\qquad\text{in}\ \Omega.
\end{align*}
\end{proof}

For the arguments below, we shall need both a localized Pohozaev-type identity associated with the dilation vector field and a boundary Pohozaev-type identity associated with a constant vector field. We record them in the following two lemmas.

\begin{lemma}\label{Lemma poho}
Let $p>1$, and $\phi\in C_c^\infty(\Omega)$.
Suppose $u \in C^2(\Omega)$ is a solution of (\ref{Mean Curvature Equation}).
Then
\begin{align*}
&\int_\Omega \frac{|\nabla u|^2}{\sqrt{1+|\nabla u|^2}} \phi
- n\int_\Omega \left(\sqrt{1+|\nabla u|^2}-1\right)\phi \notag\\
& = - \frac{n}{p+1} \int_\Omega |u|^{p+1}\phi
- \int_\Omega \frac{(x\cdot \nabla u)(\nabla u \cdot \nabla\phi)}{\sqrt{1+|\nabla u|^2}} \notag\\
&\qquad + \int_\Omega\left(\sqrt{1+|\nabla u|^2}-1\right)(x\cdot \nabla\phi)
- \frac{1}{p+1} \int_\Omega |u|^{p+1}(x\cdot \nabla\phi)
\end{align*}
\end{lemma}

\begin{proof}[Proof of Lemma \ref{Lemma poho}]

Let
\begin{align}\label{def f F}
f(t) = |t|^{p-1}t \quad \text{and} \quad F(t) = \frac{|t|^{p+1}}{p+1}.
\end{align}
Then from Lemma \ref{Lemma conservation law}, we have $\operatorname{div}T=0$.
Hence
\begin{align*}
\operatorname{div}(\phi Tx)
&= (Tx)\cdot\nabla\phi + \phi ( x\cdot \operatorname{div}T) + \phi \operatorname{tr}T \\
&= (Tx)\cdot\nabla\phi + \phi \operatorname{tr}T
\end{align*}
Integrating over $\Omega$ and using $\phi\in C_c^\infty(\Omega)$ gives
\begin{align}\label{d9f411d5a}
\int_\Omega
(Tx)\cdot\nabla\phi
+
\int_\Omega
\phi\,\operatorname{tr}T
=0.
\end{align}
Since
\begin{align*}
Tx = \frac{(\nabla u\cdot x)\nabla u}{\sqrt{1+|\nabla u|^2}} - \left( \sqrt{1+|\nabla u|^2} -1 - \frac{|u|^{p+1}}{p+1} \right) x
\end{align*}
and
\begin{align*}
\operatorname{tr}T
=\frac{|\nabla u|^2}{\sqrt{1+|\nabla u|^2}}
- n \left(\sqrt{1+|\nabla u|^2}-1-\frac{|u|^{p+1}}{p+1}\right),
\end{align*}
the stated identity follows immediately from (\ref{d9f411d5a}).
\end{proof}

\begin{lemma}\label{Lemma poho constant}
Let $X \in \mathbb{R}^n$.
Suppose $\Omega\subset \mathbb{R}^n$ is a domain with smooth boundary and $u\in C^2(\overline{\Omega})$ is a solution of (\ref{Mean Curvature Equation}).
Then
\begin{align*}
&\int_{\partial (\Omega\cap B_R)} \left(\sqrt{1+|\nabla u|^2}-1\right)\left(X\cdot \nu\right) \notag\\
&=\frac{1}{p+1} \int_{\partial (\Omega\cap B_R)} |u|^{p+1}\left(X\cdot \nu\right)
+\int_{\partial (\Omega\cap B_R)} \frac{\left(X\cdot \nabla u\right)\left(\nabla u \cdot \nu\right)}{\sqrt{1+|\nabla u|^2}},
\end{align*}
for a.e. $R>0$,
where $\nu$ denotes the outward unit normal to $\partial(\Omega \cap B_R)$.
\end{lemma}

\begin{proof}
Let $f$ and $F$ be the same as in (\ref{def f F}).
Then by Lemma \ref{Lemma conservation law} we have $\operatorname{div}T=0$.
Noting that $X$ is constant, we get
\begin{align*}
\operatorname{div}(TX)=0.
\end{align*}
Since $\partial\Omega$ is smooth, for a.e. $R>0$, $\partial B_R$ intersects $\partial\Omega$ transversally, making $\Omega\cap B_R$ a Lipschitz domain. Applying the divergence theorem on $\Omega\cap B_R$, we obtain
\begin{align*}
\int_{\partial(\Omega\cap B_R)}
(TX)\cdot\nu=0.
\end{align*}
Expanding $(TX)\cdot\nu$ gives
\begin{align*}
(TX)\cdot\nu = -\left(
\sqrt{1+|\nabla u|^2}-1-\frac{|u|^{p+1}}{p+1}
\right)(X\cdot\nu)
+\frac{(X\cdot\nabla u)(\nabla u\cdot\nu)}
{\sqrt{1+|\nabla u|^2}},
\end{align*}
which yields the desired identity.
\end{proof}

\section{The Nonexistence of Stable Solutions to the Mean Curvature Equation}

\begin{proof}[Proof of Theorem \ref{Theorem stable}]
Let $\eta \in C^\infty([0,\infty))$ be a cut-off function satisfying
\begin{align}\label{def eta}
0\leq \eta\leq 1,
\quad \eta(r) = 1 \ \text{for}\ 0\leq r\leq 1,
\quad |\eta'(r)| \leq C \ \text{for}\ 0\leq r\leq 1,
\quad \eta(r) = 0\ \text{for}\ r\geq 2,
\end{align}
For $R>0$, define
\begin{align*}
\phi_R(x) = \eta\left(\frac{|x|}{R}\right).
\end{align*}
Then
\begin{align*}
\operatorname{supp} \nabla\phi_R \subset B_{2R}\setminus B_{R},
\qquad|\nabla\phi_R(x)| \leq \frac{C}{R}.
\end{align*}
By Lemma \ref{Lemma main gradient est} with $\phi_R$,
for every
\begin{align*}
a\in\left[1, a_+(p)\right),
\end{align*}
there exists a positive integer $m=m(p,a)$ and a constant $C=C(n,p,a)$ such that
\begin{align*}
\int_{\mathbb R^n}|u|^{p+a}\phi_R^{2m}
+
\int_{\mathbb R^n}
|u|^{a-1}
\frac{|\nabla u|^2}{\sqrt{1+|\nabla u|^2}}
\phi_R^{2m}
\le
C\int_{\mathbb R^n}|\nabla\phi_R|^{\frac{2(p+a)}{p-1}}.
\end{align*}
Since $\phi_R=1$ in $B_R$, we obtain
\begin{align}\label{08ae787c1}
\int_{B_R} |u|^{p+a}
+ \int_{B_R} |u|^{a-1} \frac{|\nabla u|^2}{\sqrt{1+|\nabla u|^2}}
\leq C_{n,p,a} R^{n-\frac{2(p+a)}{p-1}}.
\end{align}

Suppose that there exists $a\in\left[1, a_+(p)\right)$ such that
\begin{align}\tag{A1}\label{A1}
\frac{2(p+a)}{p-1} > n.
\end{align}
Then the exponent on the right-hand side of (\ref{08ae787c1}) is negative.
Letting $R\rightarrow \infty$ we obtain
\begin{align*}
\int_{\mathbb{R}^n} |u|^{p+a} = 0,
\end{align*}
and hence $u \equiv 0$.

Therefore, it remains to prove that for every $1 < p < p_{JL}$,
we may choose $a\in\left[1, a_+(p)\right)$ such that (\ref{A1}) is satisfied.
The condition (\ref{A1}) is equivalent to
\begin{align*}
a > \frac{1}{2} (np -2p -n) =: a_0.
\end{align*}
As long as $a_0 < a_+(p)$ holds, the interval $(\max\{1, a_0\}, a_+(p))$ is well-defined and nonempty. Consequently, any choice of $a$ within this interval will simultaneously satisfy both $a \in (1, a_+(p))$ and condition (\ref{A1}).
Thus, it suffices to prove
\begin{align}\tag{A2}\label{A2}
a_0 < a_+(p).
\end{align}
The condition (\ref{A2}) is equivalent to
\begin{align}\tag{A3}\label{A3}
(n-6)p - (n-2) < 4\sqrt{p(p-1)}
\end{align}
We distinguish two cases.
\begin{itemize}
\item
{\bf Case $2\leq n\leq 10$.}
In this range, $p_{JL} = +\infty$,
so it suffices to prove (\ref{A3}) for all $p>1$.
We have
\begin{align*}
(n-6)p - (n-2)
& = (p-1)n -6p+2 \notag\\
&\leq (p-1)10 -6p+2 \notag\\
&= 4p-8.
\end{align*}
We split the proof into two subcases.
\begin{itemize}
\item
If $1<p<2$,
then $4p-8<0$,
and hence $(n-6)p - (n-2)<0$,
whereas $4\sqrt{p(p-1)}>0$.
So (\ref{A3}) holds trivially.
\item
If $p\geq 2$,
then
\begin{align*}
(4p-8)^2 < 16p(p-1)
\end{align*}
because this is equivalent to $p>\frac{4}{3}$.
Since both sides are nonnegative, squaring gives
\begin{align*}
4p-8<4\sqrt{p(p-1)}
\end{align*}
and hence (\ref{A3}) follows.
\end{itemize}
Therefore (\ref{A3}) holds for all $p>1$ when $2\leq n\leq 10$.

\item {Case $n\geq 11$.}
First, we claim that
\begin{align*}
&p_{JL} > \frac{n-2}{n-6}.
\end{align*}
Indeed,
\begin{align*}
p_{JL} - \frac{n-2}{n-6}
&=\frac{(n^2-8n+4+8\sqrt{n-1})(n-6) - (n-2)^2(n-10)}{(n-2)(n-10)(n-6)}\\
&=\frac{8 \left[ n+2+(n-6)\sqrt{n-1} \right]}{(n-2)(n-10)(n-6)} \\
&> 0.
\end{align*}
We split the proof into two subcases.
\begin{itemize}
\item {case $1<p<\frac{n-2}{n-6}$.}
Then the left-hand side of (\ref{A3}) is negative,
while the right-hand side is positive;
therefore, (\ref{A3}) holds immediately.

\item {case $\frac{n-2}{n-6}\leq p< p_{JL}$.}
In this case, both sides of (\ref{A3}) are nonnegative,
so (\ref{A3}) is equivalent to
\begin{align*}
\left((n-6)p - (n-2)\right)^2 < 16 p(p-1)
\end{align*}
or equivalently,
\begin{align}\label{d9323a29c}
(n-2)(n-10)p^2 -2(n^2-8n+4)p +(n-2)^2 <0
\end{align}
A direct computation shows that the discriminant $\Delta$ of the quadratic in (\ref{d9323a29c}) is
\begin{align*}
\Delta = 4(n^2-8n+4)^2-4(n-2)^3(n-10) = 256(n-1)>0
\end{align*}

Hence the quadratic has two distinct real roots,
given by
\begin{align*}
p^- := \frac{n^2-8n+4 - 8\sqrt{n-1}}{(n-2)(n-10)},
\qquad p^+ := p_{JL} = \frac{n^2-8n+4 + 8\sqrt{n-1}}{(n-2)(n-10)}.
\end{align*}
Since the quadratic is negative between its roots,
(\ref{d9323a29c}) holds for
\begin{align*}
p^- < p < p_{JL}.
\end{align*}
It remains to verify that
\begin{align*}
\frac{n-2}{n-6} > p^-
\end{align*}
Indeed,
\begin{align}\label{59f033067}
\frac{n-2}{n-6} - p^-
&= \frac{(n-2)^2(n-10) - (n-6) \left(n^2-8n+4 - 8\sqrt{n-1}\right)}{(n-6)(n-2)(n-10)} \notag\\
& = \frac{8\left[ (n-6)\sqrt{n-1} -(n+2) \right]}{(n-6)(n-2)(n-10)}.
\end{align}
Since
\begin{align*}
(n-6)^2 (n-1) - (n+2)^2 = (n-10)(n-2)^2 > 0,
\end{align*}
for $n\geq 11$,
the numerator on the right-hand side of (\ref{59f033067}) is positive.
Therefore
$\frac{n-2}{n-6} > p^-$.
Hence, for every
\begin{align*}
\frac{n-2}{n-6} \leq p <p_{JL},
\end{align*}
we have $p^- < p < p_{JL}$,
and consequently (\ref{A3}) holds.
\end{itemize}
Combining both subcases,
we have proved that (\ref{A3}) is satisfied for all $1<p<p_{JL}$ when $n\geq 11$.
\end{itemize}

Thus (\ref{A3}) holds for every $1<p<p_{JL}$, and consequently there exists
$a\in[1,a_+(p))$ satisfying (\ref{A1}). Letting $R\to\infty$ in
(\ref{08ae787c1}) then yields $u\equiv0$.
This proves Theorem \ref{Theorem stable}.
\end{proof}

\begin{proof}[Proof of Theorem \ref{Theorem stable outside compact}]
Let $\phi \in C_c^\infty(\mathbb{R}^n)$.
Using the test function $\varphi = u \phi^2$,
which is admissible by a standard density argument,
in the weak equation (\ref{mean curvature weak equation}), we get
\begin{align}\label{b22636d3b}
\int_{\mathbb{R}^n} \frac{|\nabla u|^2}{\sqrt{1+|\nabla u|^2}} \phi^2
+ 2 \int_{\mathbb{R}^n} \frac{u\phi (\nabla u\cdot \nabla\phi)}{\sqrt{1+|\nabla u|^2}}
=\int_{\mathbb{R}^n} |u|^{p+1}\phi^2
\end{align}
On the other hand, applying Lemma \ref{Lemma poho} with $\Omega = \mathbb{R}^n$ and the test function $\phi^2$ yields
\begin{align*}
&\int_{\mathbb{R}^n} \frac{|\nabla u|^2}{\sqrt{1+|\nabla u|^2}} \phi^2
- n\int_{\mathbb{R}^n} \left(\sqrt{1+|\nabla u|^2}-1\right)\phi^2 \notag\\
& = - \frac{n}{p+1} \int_{\mathbb{R}^n} |u|^{p+1}\phi^2
- 2\int_{\mathbb{R}^n}\frac{(x\cdot \nabla u)(\nabla u \cdot \nabla\phi)\phi^2}{\sqrt{1+|\nabla u|^2}} \notag\\
&\qquad + 2\int_{\mathbb{R}^n}\left(\sqrt{1+|\nabla u|^2}-1\right)(x\cdot \nabla\phi) \phi
-\frac{2}{p+1} \int_{\mathbb{R}^n} |u|^{p+1} (x\cdot \nabla\phi) \phi
\end{align*}
Combining this identity with (\ref{b22636d3b}) to eliminate the term $|u|^{p+1}\phi^2$, we obtain
\begin{align}\label{Pohozaev identity with test function}
&\left(1+\frac{n}{p+1}\right) \int_{\mathbb{R}^n} \frac{|\nabla u|^2}{\sqrt{1+|\nabla u|^2}} \phi^2
- n \int_{\mathbb{R}^n} \left(\sqrt{1+|\nabla u|^2}-1\right) \phi^2 \notag\\
&= -\frac{2n}{p+1} \int_{\mathbb{R}^n} \frac{u\phi (\nabla u\cdot \nabla\phi)}{\sqrt{1+|\nabla u|^2}}
- 2\int_{\mathbb{R}^n} \frac{(x\cdot\nabla u)(\nabla u\cdot\nabla\phi)\phi}{\sqrt{1+|\nabla u|^2}} \notag\\
&\qquad + 2\int_{\mathbb{R}^n} \left(\sqrt{1+|\nabla u|^2}-1\right)(x\cdot\nabla\phi) \phi
- \frac{2}{p+1} \int_{\mathbb{R}^n} |u|^{p+1}(x\cdot\nabla\phi) \phi
\end{align}

Next we show that the right-hand side of (\ref{Pohozaev identity with test function}) tends to $0$ by choosing a suitable $\phi=\phi_R$ as $R\rightarrow \infty$.
To this end, we estimate each term.
We choose a radial cut-off
\begin{align*}
\phi_R(x)=\eta\left( \frac{|x|}{R} \right),
\end{align*}
where $\eta\in C^\infty([0,\infty))$ is the same cut-off function defined in (\ref{def eta}).
Thus
\begin{align}\label{4daababba}
\operatorname{supp}\nabla\phi_R
\subset B_{2R}\setminus B_R,
\qquad
|\nabla\phi_R|\leq \frac{C}{R}.
\end{align}
Using the generalized H\"older inequality
\begin{align*}
\int XYZ \leq
\left(\int X^2\right)^\frac{1}{2}
\left(\int Y^{p+1}\right)^\frac{1}{{p+1}}
\left(\int Z^\gamma\right)^\frac{1}{\gamma}
\end{align*}
where
\begin{align*}
\gamma = \frac{2(p+1)}{p-1},
\qquad X = \frac{|\nabla u|}{(1+|\nabla u|^2)^{1/4}},
\qquad Y=|u|
\qquad \text{and}\qquad Z=|\nabla \phi_R|\phi_R,
\end{align*}
we obtain
\begin{align*}
&\int_{\mathbb{R}^n} \frac{u\phi_R (\nabla u\cdot \nabla\phi_R)}{\sqrt{1+|\nabla u|^2}} \notag\\
&\leq \int_{B_{2R}\setminus B_{R}} \frac{|\nabla u||u| |\nabla\phi_R|\phi_R}{(1+|\nabla u|^2)^{1/4}} \notag\\
&\leq \left(\int_{B_{2R}\setminus B_{R}} \frac{|\nabla u|^2}{\sqrt{1+|\nabla u|^2}}\right)^\frac{1}{2}
\left(\int_{B_{2R}\setminus B_{R}} |u|^{p+1}\right)^{\frac{1}{p+1}}
\left(\int_{B_{2R}\setminus B_{R}} |\nabla\phi_R|^\gamma\phi_R^\gamma \right)^{\frac{1}{\gamma}}
\end{align*}
For $n>2$, since $p=(n+2)/(n-2)$, we have $\gamma=n$.
Therefore, using (\ref{4daababba}) and $|B_{2R}\setminus B_R|\le CR^n$ we have
\begin{align*}
\left(
\int_{B_{2R}\setminus B_R}
|\nabla\phi_R|^\gamma\phi_R^\gamma
\right)^{1/\gamma}
\le
C R^{\frac n\gamma-1}
\le C,
\end{align*}
Consequently,
\begin{align}\label{921433a99}
\left| \int_{\mathbb{R}^n} \frac{u\phi_R(\nabla u\cdot\nabla\phi_R)} {\sqrt{1+|\nabla u|^2}}\right|
\leq C\left( \int_{B_{2R}\setminus B_R} \frac{|\nabla u|^2} {\sqrt{1+|\nabla u|^2}} \right)^{1/2}
\left( \int_{B_{2R}\setminus B_R}|u|^{p+1} \right)^{\frac1{p+1}}.
\end{align}
Moreover, noting that $x$ is comparable to $R$ in $B_{2R}\setminus B_R$,
we have
\begin{align}\label{7bbdf943b}
\left| \int_{\mathbb{R}^n} \frac{(x\cdot\nabla u)(\nabla u\cdot\nabla\phi_R)\phi_R}{\sqrt{1+|\nabla u|^2}} \right|
\leq C \int_{B_{2R}\setminus B_{R}} \frac{|\nabla u|^2}{\sqrt{1+|\nabla u|^2}}
\end{align}
Using the elementary inequality
\begin{align}\label{elementary inequality}
\sqrt{1+t}-1\leq \frac{t}{\sqrt{1+t}} \qquad \forall t\geq 0,
\end{align}
we get
\begin{align}\label{00033cd9f}
\left| \int_{\mathbb{R}^n} \left(\sqrt{1+|\nabla u|^2}-1\right)(x\cdot\nabla\phi_R)\phi_R \right|
\leq C \int_{B_{2R}\setminus B_{R}} \frac{|\nabla u|^2}{\sqrt{1+|\nabla u|^2}}.
\end{align}
Similarly,
\begin{align}\label{b3f00123a}
\left| \int_{\mathbb{R}^n} |u|^{p+1}(x\cdot\nabla\phi_R) \phi_R \right|
\leq C \int_{B_{2R}\setminus B_{R}} |u|^{p+1}.
\end{align}
Since $p = p_S$ and $u$ is stable outside a compact set $\mathcal K\subset\mathbb R^n$,
using Lemma \ref{Lemma integrable}(i) we deduce that $|\nabla u|^2/\sqrt{1+|\nabla u|^2}$ and $|u|^{p+1}$ are integrable near infinity.
Hence the right-hand sides of (\ref{921433a99}) -- (\ref{b3f00123a}) tend to $0$ as $R\rightarrow \infty$.
If we replace $\phi$ in (\ref{Pohozaev identity with test function}) with $\phi_R$,
then the right-hand side of (\ref{Pohozaev identity with test function}) will tend to $0$ as $R \rightarrow \infty$.
Thus, letting $R\rightarrow \infty$ gives
\begin{align}\label{e315d48f6}
\left(1+\frac{n}{p+1}\right) \int_{\mathbb{R}^n} \frac{|\nabla u|^2}{\sqrt{1+|\nabla u|^2}} = n \int_{\mathbb{R}^n} \left(\sqrt{1+|\nabla u|^2}-1\right)
\end{align}
Define
\begin{align*}
A := 1+\frac{n}{p+1},
\qquad f(t) := (n-A)t + \frac{A}{t} - n, \quad t > 0.
\end{align*}
Then (\ref{e315d48f6}) is equivalent to
\begin{align}\label{d4b6b6633}
\int_{\mathbb{R}^n} f(\sqrt{1+|\nabla u|^2}) = 0
\end{align}
Since $n \ge 3$ and $p = p_S = (n+2)/(n-2)$, we have $A=n/2$.
Consequently,
\begin{align*}
f(t) = \frac{n(t-1)^2}{2t} \geq 0.
\end{align*}
with equality if and only if $t=1$.
Therefore from (\ref{d4b6b6633}) we obtain that
\begin{align*}
\sqrt{1+|\nabla u|^2}=1
\qquad\text{in}\ \mathbb{R}^n,
\end{align*}
and therefore
\begin{align*}
\nabla u\equiv 0\qquad\text{in}\ \mathbb{R}^n.
\end{align*}
It follows that $u$ is constant.
Substituting this into (\ref{Mean Curvature Equation}) gives $|u|^{p-1}u=0$,
and hence $u\equiv0$,
which completes the proof of Theorem \ref{Theorem stable outside compact}.
\end{proof}

\begin{proof}[Proof of Theorem \ref{Theorem epigraph}]
Let $R>0$.
We apply the Pohozaev-type identity from Lemma \ref{Lemma poho constant} to the domain $\Omega \cap B_R$,
\begin{align}\label{030fb9b7b}
\int_{\partial(\Omega \cap B_R )} \left(\sqrt{1+|\nabla u|^2}-1\right)\left(X\cdot \nu\right)
&=\frac{1}{p+1} \int_{\partial(\Omega \cap B_R )} |u|^{p+1}\left(X\cdot \nu\right) \notag\\
&\qquad +\int_{\partial(\Omega \cap B_R )} \frac{\left(X\cdot \nabla u\right)\left(\nabla u \cdot \nu\right)}{\sqrt{1+|\nabla u|^2}}
\end{align}
Here $\nu$ denotes the outward unit normal to $\partial(\Omega \cap B_R)$.
Now decompose $\partial(\Omega \cap B_R ) = (\partial \Omega \cap B_R) \cup (\Omega \cap \partial B_R)$.
Noting $u=0$ on $\partial \Omega$,
we have
\begin{align*}
\int_{\partial\Omega \cap B_R} |u|^{p+1}\left(X\cdot \nu\right) = 0.
\end{align*}
Moreover,
on $\partial \Omega$,
since $u=0$,
the gradient $\nabla u$ is parallel to the normal $\nu$.
Therefore
\begin{align*}
\int_{\partial\Omega \cap B_R} \frac{\left(X\cdot \nabla u\right)\left(\nabla u \cdot \nu\right)}{\sqrt{1+|\nabla u|^2}}
=\int_{\partial\Omega \cap B_R} \left(X\cdot \nu\right)\frac{|\nabla u|^2}{\sqrt{1+|\nabla u|^2}}
\end{align*}
Thus, from (\ref{030fb9b7b}), we have
\begin{align*}
&\int_{\partial \Omega \cap B_R} \left(\frac{|\nabla u|^2}{\sqrt{1+|\nabla u|^2}} -\sqrt{1+|\nabla u|^2} +1\right) \left(X\cdot \nu\right) \notag\\
&= \int_{\Omega \cap \partial B_R} \left(\sqrt{1+|\nabla u|^2}-1\right)\left(X\cdot \nu\right)
-\frac{1}{p+1} \int_{\Omega \cap \partial B_R} |u|^{p+1}\left(X\cdot \nu\right) \notag\\
&\qquad -\int_{\Omega \cap \partial B_R} \frac{\left(X\cdot \nabla u\right)\left(\nabla u \cdot \nu\right)}{\sqrt{1+|\nabla u|^2}}
\end{align*}
Absorbing $X$ and $\nu$ into a constant $C$ and applying elementary inequality (\ref{elementary inequality}) to the right-hand side of the above identity,
we obtain
\begin{align}\label{0ba95f35e}
&\int_{\partial \Omega \cap B_R} \left(\frac{|\nabla u|^2}{\sqrt{1+|\nabla u|^2}} -\sqrt{1+|\nabla u|^2} +1\right) \left(X\cdot \nu\right) \notag\\
&\leq C \int_{\Omega \cap \partial B_R} \left(\frac{|\nabla u|^2}{\sqrt{1+|\nabla u|^2}} +|u|^{p+1}\right).
\end{align}
Since $1<p\leq p_S$ and $u$ is stable outside a compact set $\mathcal K\subset \Omega$,
by Lemma \ref{Lemma integrable},
we obtain
$|u|^{p+1}$, $|\nabla u|^2 / (\sqrt{1+|\nabla u|^2)}$ are integrable near infinity.
Consequently, there exists a sequence $R_k\rightarrow \infty$ such that
\begin{align*}
\int_{\Omega \cap \partial B_{R_k}} \left(\frac{|\nabla u|^2}{\sqrt{1+|\nabla u|^2}} +|u|^{p+1}\right) \rightarrow 0
\end{align*}
Taking $R=R_k$ in (\ref{0ba95f35e}) and letting $k\rightarrow \infty$,
we obtain
\begin{align}\label{6159a508d}
\int_{\partial \Omega} \left(\frac{|\nabla u|^2}{\sqrt{1+|\nabla u|^2}} -\sqrt{1+|\nabla u|^2} +1\right) \left(X\cdot \nu\right) = 0
\end{align}
Noting $X\cdot \nu\geq 0$ on $\partial \Omega$ and using (\ref{elementary inequality}) again,
we conclude from (\ref{6159a508d}) that
\begin{align}\label{c07f21d6e}
\left(\frac{|\nabla u|^2}{\sqrt{1+|\nabla u|^2}} -\sqrt{1+|\nabla u|^2} +1\right)\left(X\cdot \nu\right) = 0 \qquad \text{on}\ \partial\Omega
\end{align}
Because $X\cdot\nu\not\equiv 0$ and $\partial\Omega$ is smooth,
there exists a relatively open subset $\Gamma\subset \partial \Omega$ on which $X \cdot \nu>0$.
Then (\ref{c07f21d6e}) implies
\begin{align*}
\frac{|\nabla u|^2}{\sqrt{1+|\nabla u|^2}} -\sqrt{1+|\nabla u|^2} +1 = 0 \qquad \text{on}\ \Gamma,
\end{align*}
which is equivalent to $\nabla u = 0$ on $\Gamma$.
Hence we get both $ u=0$ and $\nabla u = 0$ on $\partial\Omega\cap \Gamma$.
In the following, we show that such an overdetermined boundary condition forces $u\equiv 0$ in $\Omega$.
We use a unique continuation argument.
Rewrite the original equation as a linear elliptic equation
\begin{align}\label{linearized equation}
\operatorname{div}\left(A(x) \nabla u\right) + V(x) u = 0.
\end{align}
where
\begin{align*}
A(x) := \frac{I}{\sqrt{1+|\nabla u|^2}},
\qquad V(x) = |u|^{p-1}.
\end{align*}
Fix a point $y\in \Gamma$.
By reducing $\delta>0$ if necessary, we may assume that
\begin{align*}
\partial\Omega\cap B_\delta(y)\subset\Gamma.
\end{align*}

Choose $x_0\in\Omega\cap B_{\delta/2}(y)$.
Since $\Omega$ is connected, for any fixed $z\in\Omega$, there exists a continuous path $\gamma$ in $\Omega$ joining $x_0$ to $z$.
We may therefore choose a bounded connected open set $D\subset\Omega$ containing both $x_0$ and $z$.
Set
$$
U:=D\cup B_\delta(y).
$$
Then $U$ is a bounded connected open set.
Define the zero extension of $u$ to $U$ by
\begin{align*}
\widetilde{u} =
\left\{
\begin{array}{ll}
u, & x \in U \cap \Omega, \\
0, & x \in U \setminus \Omega.
\end{array}
\right.
\end{align*}

We next show that $\widetilde u$ is $C^2$ across $\Gamma$.
Since $u\in C^2(\overline{\Omega})$ and both $u$ and $\nabla u$ vanish on $\Gamma$,
all first-order tangential derivatives and all second-order derivatives involving at least one tangential direction vanish on $\Gamma$.
To show that the normal second derivative $u_{\nu\nu}$ also vanishes, we expand the original equation.
Using the product rule, the equation can be written as
\begin{align}\label{expanded equation}
-\nabla g \cdot \nabla u - g \Delta u = |u|^{p-1}u,
\qquad \text{where}\ g(x) := \frac{1}{\sqrt{1+|\nabla u|^2}}.
\end{align}
As $x$ approaches the boundary $\Gamma$,
we have $u \to 0$ and $\nabla u \to 0$.
The first term $\nabla g \cdot \nabla u$ vanishes because $\nabla g$ is bounded near $\Gamma$
(as $u$ is $C^2$ and $g$ is a smooth function of the bounded quantity $|\nabla u|$) and $\nabla u \rightarrow 0$.
Meanwhile, $g \to 1$ and $|u|^{p-1}u \to 0$.
Therefore, passing to the limit $x \to \Gamma$ in (\ref{expanded equation}),
the equation reduces to
\begin{align*}
\Delta u = 0 \qquad \text{on}\ \Gamma.
\end{align*}
It follows that $u_{\nu\nu}=0$ on $\Gamma$ and therefore $D^2u=0$ on $\Gamma$.
This ensures that the zero-extension $\widetilde{u}$ is $C^2$ across the boundary $\Gamma$.
Therefore, $\widetilde{u} \in C^2(\overline{U})$, and it satisfies the linear equation
\begin{align*}
\operatorname{div}\left(\widetilde{A}(x) \nabla \widetilde{u}\right) + \widetilde{V}(x) \widetilde{u} = 0 \qquad \text{in}\ U,
\end{align*}
where $\widetilde{A}(x)$ and $\widetilde{V}(x)$ are defined by replacing $u$ with $\widetilde{u}$ in the definitions of $A(x)$ and $V(x)$.

We now verify the hypotheses of the unique continuation theorem of Garofalo--Lin (Theorem 1.1 in \cite{Garofalo Lin CPAM 1987}) on the connected domain $U$:

\begin{itemize}
\item[1.]
{\bf Uniform ellipticity:}
Since $\widetilde{u} \in C^2(\overline{U})$,
there exists $M>0$ such that $|\nabla \widetilde{u}| \leq M$.
For every $\xi\in\mathbb{R}^n$,
\begin{align*}
\xi^{T}\widetilde{A}(x)\xi = \frac{|\xi|^2}{\sqrt{1+|\nabla \widetilde{u}|^2}}.
\end{align*}
Therefore
\begin{align*}
\frac{1}{\sqrt{1+M^2}} |\xi|^2 \leq \xi^{T}\widetilde{A}(x)\xi \leq |\xi|^2.
\end{align*}
This shows that $\widetilde{A}$ is uniformly elliptic.

\item[2.]
{\bf Lipschitz continuity of $A$:}
Since $\widetilde{u}\in C^2(\overline{U})$,
we have $\nabla \widetilde{u} \in C^1(\overline{U})$.
The map
\begin{align*}
t \mapsto \frac{1}{\sqrt{1+|t|^2}}
\end{align*}
is smooth.
Consequently, $\widetilde{A}(x) \in C^1(\overline{U})$,
and $\widetilde{A}(x)$ is Lipschitz on the compact set $\overline{U}$.

\item[3.]
{\bf Boundedness of $V$:}
Since $u\in C^2(\overline{U})$,
$V(x)$ is bounded on $\overline{U}$.
The condition (iii) in Garofalo--Lin \cite{Garofalo Lin CPAM 1987} is trivially satisfied for bounded potentials.
\end{itemize}
Therefore, the operator in (\ref{linearized equation}) satisfies the structural assumptions (i), (ii), and (iii) of Theorem 1.1 in \cite{Garofalo Lin CPAM 1987} for $n>2$.
By that theorem, the equation has the strong unique continuation property.
Since $\widetilde{u}\equiv 0$ in the nonempty open set $U \setminus \overline{\Omega}$,
we must have $\widetilde{u}\equiv 0$ in $U$, and hence $u(z) = 0$.
Since $z \in \Omega$ can be chosen arbitrarily,
we obtain $u\equiv 0$ in the whole connected domain $\Omega$.
This completes the proof of Theorem \ref{Theorem epigraph}.
\end{proof}

\textbf{Conflict of Interest Statement.}
The author declares that there is no conflict of interest.

\textbf{Data Availability Statement.}
No datasets were generated or analysed during the current study.


\begin{thebibliography}{99}
\bibitem{Caffarelli Gidas Spruck CPAM 1989} L.~\'A. Caffarelli, B. Gidas and J. Spruck, Asymptotic symmetry and local behavior of semilinear elliptic equations with critical Sobolev growth, Comm. Pure Appl. Math. {\bf 42} (1989), no.~3, 271--297; MR0982351
\bibitem{Damascelli Farina Sciunzi Valdinoci AIHP 2009} L. Damascelli, A. Farina, B. Sciunzi and E. Valdinoci, Liouville results for $m$-Laplace equations of Lane-Emden-Fowler type, Ann. Inst. H. Poincar\'e{} C Anal. Non Lin\'eaire {\bf 26} (2009), no.~4, 1099--1119; MR2542716
\bibitem{del Pino Guerra JDE 2007} M.~A. del~Pino and I.~A. Guerra, Ground states of a prescribed mean curvature equation, J. Differential Equations {\bf 241} (2007), no.~1, 112--129; MR2356212
\bibitem{Esteban Lions PRSE 1982} M.~J. Esteban and P.-L. Lions, Existence and nonexistence results for semilinear elliptic problems in unbounded domains, Proc. Roy. Soc. Edinburgh Sect. A {\bf 93} (1982/83), no.~1-2, 1--14; MR0688279
\bibitem{Farina JMPA 2007} A. Farina, On the classification of solutions of the semilinear equation on unbounded domains of $\Bbb R^N$, J. Math. Pures Appl. (9) {\bf 87} (2007), no.~5, 537--561; MR2322150
\bibitem{Farina Mercuri Willem CVPDE 2019} A. Farina, C. Mercuri and M. Willem, A Liouville theorem for the $p$-Laplacian and related questions, Calc. Var. Partial Differential Equations {\bf 58} (2019), no.~4, Paper No. 153, 13 pp.; MR3989957
\bibitem{Farina Navarro DCDS 2020} A. Farina and M.~A. Navarro, Some Liouville-type results for stable solutions involving the mean curvature operator: the radial case, Discrete Contin. Dyn. Syst. {\bf 40} (2020), no.~2, 1233--1256; MR4043671
\bibitem{Farina Sciunzi Valdinoci ASNSPS 2008} A. Farina, B. Sciunzi and E. Valdinoci, Bernstein and De Giorgi type problems: new results via a geometric approach, Ann. Sc. Norm. Super. Pisa Cl. Sci. (5) {\bf 7} (2008), no.~4, 741--791; MR2483642
\bibitem{Farina Sciunzi Vuono NA 2024} A. Farina, B. Sciunzi and D. Vuono, Liouville-type results for some quasilinear anisotropic elliptic equations, Nonlinear Anal. {\bf 238} (2024), Paper No. 113402, 15 pp.; MR4650838
\bibitem{Filippucci Zheng MathAnn 2024} R. Filippucci and Y. Zheng, Existence and nonexistence of solutions for the mean curvature equation with weights, Math. Ann. {\bf 390} (2024), no.~4, 5821--5861; MR4816123
\bibitem{Garofalo Lin CPAM 1987}N. Garofalo and F. Lin, Unique continuation for elliptic operators: a geometric-variational approach, Comm. Pure Appl. Math. {\bf 40} (1987), no.~3, 347--366; MR0882069
\bibitem{Gidas Spruck CPAM 1981} B. Gidas and J. Spruck, Global and local behavior of positive solutions of nonlinear elliptic equations, Comm. Pure Appl. Math. {\bf 34} (1981), no.~4, 525--598; MR0615628
\bibitem{Joseph Lundgren ARMA 1973} D.~D. Joseph and T.~S. Lundgren, Quasilinear Dirichlet problems driven by positive sources, Arch. Rational Mech. Anal. {\bf 49} (1972/73), 241--269; MR0340701
\bibitem{Ni Serrin CPAM 1986} W.-M. Ni and J.~B. Serrin Jr., Nonexistence theorems for singular solutions of quasilinear partial differential equations, Comm. Pure Appl. Math. {\bf 39} (1986), no.~3, 379--399; MR0829846
\bibitem{Ni Serrin RCMP 1985} W.-M. Ni and J.~B. Serrin Jr., Nonexistence theorems for quasilinear partial differential equations, Rend. Circ. Mat. Palermo (2) Suppl. No. 8 (1985), 171--185; MR0881397
\bibitem{Pucci Serrin Indiana 1986} P. Pucci and J.~B. Serrin Jr., A general variational identity, Indiana Univ. Math. J. {\bf 35} (1986), no.~3, 681--703; MR0855181
\bibitem{Veron Pohozaev JAM 2001} M.-F. Bidaut-V\'eron and S.~I. Pohozaev, Nonexistence results and estimates for some nonlinear elliptic problems, J. Anal. Math. {\bf 84} (2001), 1--49; MR1849197
\end{thebibliography}
\end{document}